\documentclass{article}
\usepackage[utf8]{inputenc}
\usepackage[shortlabels]{enumitem}
\usepackage[thicklines]{cancel}
\usepackage{amsfonts,amsmath,amsthm, amssymb, mathtools, mathrsfs, accents, nicefrac,eucal}
\usepackage[dvipsnames]{xcolor}
\usepackage{tikz, relsize, scalerel, stackengine, graphicx}
\usepackage{soul}

\usepackage{geometry}
\theoremstyle{plain}
\newtheorem{thm}{Theorem}[section]
\newtheorem{lemma}[thm]{Lemma}

\newtheorem{coroll}[thm]{Corollary}
\theoremstyle{remark}
\newtheorem{rmk}[thm]{Remark}
\theoremstyle{definition}
\newtheorem{defn}[thm]{Definition}

\newtheorem{ldefn}[thm]{Lemma-definition}

\allowdisplaybreaks

\newcommand{\mr}{\mathrm}

\newcommand{\mbb}{\mathbb}

\newcommand{\efer}{E^{\mr{fer}}}
\newcommand{\eaf}{E^{\mr{af}}}

\definecolor{bluetto}{RGB}{22, 96, 186}
\definecolor{rossino}{RGB}{186, 22, 47}

\usepackage{xparse}

\newcommand\Ccancel[2][gray]{
	\let\OldcancelColor\CancelColor
	\renewcommand\CancelColor{\color{#1}}
	\cancel{#2}
	\renewcommand\CancelColor{\OldcancelColor}
}
\newcommand{\lroungle}{\mathrel{\mkern-4mu 
		\raisebox{-2.1pt}{%
			\tikz[line cap=round, line join=round]
			\draw
			(0ex, 2.2ex) to[out=-135,in=135, distance=7] 	 (0ex,0ex)
			;%
		}\mkern-3mu}
}
\newcommand{\rroungle}{%
	\mathrel{\mkern-4mu\raisebox{-2.1pt}{%
			\tikz[line cap=round, line join=round]
			\draw
			(0ex, 2.2ex) to[out=-45,in=45, distance=7] 	 (0ex,0ex)
			;%
	}}
}
\DeclareDocumentCommand\expct{s m}
{
	\IfBooleanTF{#1}
	{\mkern0mu\lroungle\mkern0mu #2 \mkern0mu\rroungle\mkern0mu}
	{\stretchleftright{\mkern1.5mu\lroungle\mkern4mu}{#2\vphantom{\raisebox{-.4pt}{\scalebox{1.2}{$#2$}} }}{\mkern4mu\rroungle\mkern0mu}}
}
\DeclareDocumentCommand\expctscrpt{s m}
{
	\IfBooleanTF{#1}
	{\mkern0mu\lroungle\mkern0mu #2 \mkern0mu\rroungle\mkern0mu}
	{\stretchleftright{\mkern2.5mu\lroungle\mkern2.5mu}{#2\vphantom{\raisebox{-.4pt}{\scalebox{0.9}{$#2$}} }}{\mkern2.5mu\rroungle\mkern-4mu}}
}
\DeclareDocumentCommand\brac{ s m t\chet s g }
{ 
	\IfBooleanTF{#3}
	{ 
		\IfBooleanTF{#1}
		{ 
			\IfNoValueTF{#5}
			{\bracchet*{#2}{} \IfBooleanTF{#4}{*}{}}
			{\bracchet*{#2}{#5}}
		}
		{
			\IfBooleanTF{#4}
			{ 
				\IfNoValueTF{#5}
				{\bracchet{#2}{} *}
				{\bracchet*{#2}{#5}}
			}
			{\bracchet{#2}{\IfNoValueTF{#5}{}{#5}}} 
		}
	}
	{ 
		\IfBooleanTF{#1}
		{\lroungle \smash{#2} \rvert}
		{\stretchleftright{\lroungle\mkern2.5mu}{#2}{\rvert}}
		\IfBooleanTF{#4}{*}{}
		\IfNoValueTF{#5}{}{#5}
	}
}
\DeclareDocumentCommand\chet{ s m }
{ 
	\IfBooleanTF{#1}
	{\vphantom{#2}\lvert\smash{#2}\rroungle} 
	{\stretchleftright{\lvert}{#2}{\mkern2.5mu\rroungle}} 
}

\DeclareDocumentCommand\prodottointerno{ s m g }
{ 
	\IfBooleanTF{#1}
	{ 
		\IfNoValueTF{#3}
		{\vphantom{#2}\lroungle\smash{#2}\vert\smash{#2}\rroungle}
		{\vphantom{#2#3}\lroungle\smash{#2\mkern1mu}\vert\smash{\mkern.5mu#3}\rroungle}
	}
	{ 
		\IfNoValueTF{#3}
		{\stretchleftright{\lroungle}{#2}{\vert} #2 \stretchrel*{\rroungle}{\vphantom{#2}}}
		{\stretchleftright{\lroungle}{\mkern4mu #2 \vphantom{#3}}{\vert} {#3\mkern4.5mu}\stretchrel*{\rroungle}{\vphantom{#2#3}}}
	}
}
\DeclareDocumentCommand\bracchet{}{\prodottointerno}

\newcommand\restr[2]{{
		\left.\kern-\nulldelimiterspace 
		#1 
		\vphantom{\Big|} 
		\right|_{#2} 
}}

\title{Modulated phases in Ising systems with long-range antiferromagnetic and short-range ferromagnetic interactions}
\author{Andrea Braides\footnote{present address: Department of Mathematics, University of Rome Tor Vergata, via della ricerca scientifica 1, Rome, Italy } \ and Fabrizio Caragiulo \\
SISSA, via Bonomea 265, Trieste, Italy}
\date{}

\begin{document}

\maketitle

\section{Introduction}
We consider large spin systems with short-range ferromagnetic interactions and long-range antiferromagnetic interactions subjected to periodic boundary conditions. Such systems are modeled by an energy of the form
\begin{equation}\label{energy}\begin{split}     
        E_N(\underline{\sigma})&\coloneqq \overbrace{-J \sum_{i=1}^N\sigma_i\sigma_{i+1}}^{\efer_N(\underline{\sigma})} +\overbrace{\sum_{i=1}^N\sum_{j\in \mbb Z,j\neq i}\frac{\sigma_i\sigma_j}{|i-j|^p}}^{\eaf_N(\underline{\sigma})},
    \end{split}\end{equation}
    with $p>1$ and $J>0$,
with a ferromagnetic part $\efer_N$ interacting with an antiferromagnetic part $\eaf_N$, depending on a $N$-periodic spin state $\underline{\sigma}:\mbb Z\to\{-1,1\}$ such that $\sigma_{i+N}=\sigma_i$. 
It has been shown by Giuliani, Lebowitz and Lieb \cite{GLL} that, under suitable conditions on $J$ and $p$, as the size of the system $N$ diverges, minimizers of the system tend to alternate groups of $1$ and $-1$ of the same length $h^\star $ (or, in some exceptional cases of lengths either $h^\star $ or $h^\star +1$) for some unique $h^\star $ determined by $p$ and $J$. Hence, if we let $e(h^\star )$ denote the energy per site of this periodic state when $N=2h^\star $, then we have that
\begin{equation}\label{energy-1}
\min E_N= N e(h^\star ) + O(1),
\end{equation}
the remainder being $0$ when $N$ is a multiple of $h^\star $. 

In this paper we give an asymptotic description of the states $\underline{\sigma}$ such that $E_N(\underline{\sigma})- N e(h^\star )=O(1)$, by computing the $\Gamma$-limit of these normalized energies 
\begin{equation}\label{energy-2}
F_N(\underline{\sigma})=E_N(\underline{\sigma})- N e(h^\star ).
\end{equation}
To that end, we first prove that if $F_N(\underline{\sigma}^N)$ is equibounded then the domain of  $\underline{\sigma}^N$ can be decomposed into alternate arrays where $\underline{\sigma}^N$ take the value $1$ and $-1$ all of length $h^\star $ except for a number of indices equibounded with $N$. Hence, locally such $\underline{\sigma}^N$ will be a translation of the $2h^\star $-periodic function $\underline{\sigma}^\star $ with $\underline{\sigma}^\star _i=-1$ if $i\in\{1,\ldots, h^\star \}$ and $\underline{\sigma}^\star _i=1$ if $i\in\{h^\star +1,\ldots, 2h^\star \}$.

 In order to describe the asymptotics of such systems, it is convenient to scale the domain $\mbb Z$ by $\frac1N$, so that each $\underline{\sigma}^N$ is identified as a $1$-periodic spin function defined on $\frac1N\mbb Z$. Up to considering subsequences, we can find a finite set of points $x_1,\ldots,x_K\in(0,1]$ and integers $r_1,\ldots,r_{K+1}\in\{1,\ldots, 2h^\star \}$ with $r_k\neq r_{k+1}$, such that, setting $x_{K+1}=x_1+1$, if $\frac{i}N$ lies in a compact interval of $(x_k, x_{k+1})$ we have $\underline{\sigma}^N_i=\underline{\sigma}^\star _{i+r_k}$. In this case, we define the limit of $\underline{\sigma}^N$ as the $1$-periodic piecewise-constant function $r$ such that $r(x)= r_k$ if $x\in (x_k, x_{k+1})$. This description highlights that a finite number of modulated phases obtained as a translation of the `absolute' ground state $\underline{\sigma}^\star $ can coexists with a bounded extra energy from the minimal one. The energy transition between two neighbouring modulated phases can actually be computed in terms of the difference of the corresponding translations so that we may characterize a function $\phi$ such that the energies $F_N$ $\Gamma$-converge to the functional
 \begin{equation}\label{energy-3}
F_\infty(r)=\sum_{x\in J(r)\cap (0,1]} \phi(r(x^+)-r(x^-)),
\end{equation}
defined on $1$-periodic functions $r$ that are piecewise constant on bounded intervals, where $J(r)$ denotes the set of discontinuity points of $r$. The function $\phi (j)$ is the minimal energy density of a defect of size $j$ (modulo $2h^\star $), and is simply defined as the limit for $N\to+\infty$ and $N=j$ modulo $2h^\star $ of the renormalized minimal energies $\min E_N- Ne(h^\star )$.

For an analogous variational justification of modulated phases for antiferromagnetic systems with long, but not infinite, range we refer to \cite[Chapter 7]{ABCS}. In analogy with that analysis, for lattice systems in dimension $d$ higher than one we expect a limit description with $2dh^\star $ parameters and a limit partition of the reference set into sets of finite perimeter; e.g.~in dimension $2$ we expect a limit description with $4h^\star $ parameters, with $2h^\star $ parameters for vertical and $2h^\star $ parameters for horizontal stripes. The boundaries between two sets parameterized by two variants of horizontal (or two variants of vertical) stripes correspond to the one-dimensional anti-phase boundaries, while boundaries between two sets parameterized by stripes of different directions correspond to phase boundaries between different textures. We note that energies on partitions into sets of finite perimeter are much more complex than perimeter functionals since they need to satisfy $BV$-ellipticity conditions \cite{AB2}. We refer to \cite{ABC} for a direct computation of a partition energy derived from next-to-nearest neighbour antiferromagnetic interactions. In this perspective, it may be interesting to extend our study to higher-dimensional versions of \cite{GLL}, such as in \cite{GS}, or continuum approximations as in \cite{DR}, where the number of parameters may reduce to only $d$, corresponding to the different orientation of stripes. For a study of one-dimensional systems of nonlinear elastic lattice interactions exhibiting antiphase boundaries we refer to \cite{BCST} (see also \cite{BVZ})

\section{Preliminaries: properties of ground states}

Let $p>1$ and $ J>0$ be fixed. 
We will  consider the Hamiltonian $E_N$ in \eqref{energy} as defined on spin states on 
 some finite lattice $\Lambda_N=\mbb Z/N\mbb Z$,  that is on $\underline \sigma  \in \{\pm 1\}^{\Lambda_N}$.
 In this case, we may write
    \[\begin{split}
        E_N(\underline{\sigma})&\coloneqq \overbrace{-J \sum_{i\in \Lambda_N}\sigma_i\sigma_{i+1}}^{\efer_N(\underline{\sigma})} +\overbrace{\sum_{i,j\in \Lambda_N, i\ne j}\sum_{n\in \mbb Z}\frac{\sigma_i\sigma_j}{|i-j+nN|^p}+\frac{2}{N^{p-1}} \sum_{n= 1}^\infty\frac{1}{n^p}
       } ^{\eaf_N(\underline{\sigma})}.
       \color{red} 
    \end{split}\]
The last term is the part of the energy deriving from the sum of the interactions of each site $i\in\{1,\ldots,N\} $ with the sites $j$ with $j-i\in N\mbb Z\setminus\{0\}$, which is independent of $i$ since $\sigma_i\sigma_j=1$ if $j-i\in N\mbb Z$
independently of the sign of $\sigma_i$. Since this part is independent of $\underline\sigma$ it does not influence the minimization of $E_N$; however, it is convenient to maintain it for the asymptotic analysis of $E_N$. Indeed, if $\underline \sigma  \in \{\pm 1\}^{\Lambda_N}$ and $M\in\mathbb N$, then by definition we have
\begin{equation}\label{5}
E_{MN}(\underline \sigma)= ME_{N}(\underline \sigma),
\end{equation}
where in the first term $\underline \sigma$ (or more precisely, its $N$-periodic extension) is considered as an element of $ \{\pm 1\}^{\Lambda_{MN}}$. Note however that the difference between the definitions of the energies with or without the last term vanishes uniformly as $N\to+\infty$.

In this section we gather some results of Giuliani, Lebowitz and Lieb \cite{GLL} on minimizers for $E_N$.

\subsection{Representation in terms of arrays of equal sign}
Each spin state $\underline{\sigma}$ can be  identified, up to a discrete translation $\tau\in \mbb Z/N\mbb Z$, with a minimal array of lengths of its sets of consecutive equal spins (taking into account periodicity); e.g. if $N=8$ and \[\underline{\sigma}\colon (1,2,\ldots,8)\mapsto (1,1,-1,-1,1,-1,-1,1)\] then $\underline{\sigma}$ is equivalent, in terms of the energy, to \[\underline{\sigma}^0\colon(1,2,\ldots,8)\mapsto (1,1,1,-1,-1,1,-1,-1),\] after a translation of $\tau$ (that is, $\sigma_i=\sigma^0_{i+\tau}$) with $\tau=1$, and hence is described by four sets of size $3,2,1,2$ respectively and by the translation $1$. The original subdivision of $\underline{\sigma}$ in sets of sizes $2,2,1,2,1$ is made of five sets (and not four) and hence is not minimal.  Since translations do not affect the energy, we can then rewrite $E_N$ as a Hamiltonian on  such arrays instead of spins; with an abuse of notation, we write
\[ E_N(\underline{h})=E_N({\underline\sigma})\]
if $\underline{h}=(h_1,\ldots,h_M)$ is such that $h_1+\ldots+h_M=N$ and ${\sigma}$ is any spin state with  $\underline{h}$ as the corresponding array (up to translations). 
We note that if $M> 1$, then $M$ is even.

Given $h\in\mathbb N$ we can consider the configuration with an array of $1$s  of length $h$ one of $-1$s of equal length with energy $E_{2h}(h,h)$. We define the \emph{$h$-periodic energy per site} as
    \begin{equation} \label{enpersite}
     e(h)\coloneqq \frac{ E_{2h}(h,h)}{2h}
\end{equation}
If we consider the (periodic) configuration with arrays of $1$s and $-1$s of equal length $h$ for all even $M$ the energy $E_{Mh}(h\ldots,h)$ is proportional to $M$ by \eqref{5}, 
so that 
    \begin{equation}
    e(h)=\frac{ E_{Mh}(h,\ldots,h)}{Mh} 
\end{equation}
for all (even) $M$.

The following lemma describes the optimal $h$ for the  $h$-periodic energy per site. If $p>2$ we set
    \begin{equation}
J_p\coloneqq \frac{1}{\Gamma(p)}\int_0^{+\infty}\frac{\alpha^{p-1 }\mr e^{-\alpha}}{(1-\mr e^{-\alpha})^2} {\mr d\alpha}
\end{equation}

\begin{lemma} \label{lemma:eh} If $1<p\le 2$ and $J$ is arbitrary or if $p>2$ and $J<J_p$, then
     $e$ attains its minimum on $\mbb N$  at most two different points $h^\star , h^\star  +1$.
     If $p>2$ and $J>J_p$ then $e$ is always decreasing.
\end{lemma}

\subsection{Asymptotic description of ground states}
The following theorem describes the behaviour of ground states for parameters for which $e$ has a unique minimizer. A slightly more complex statement holds when $e$ has two minimizers.

\begin{thm}[Ground state energy asymptotics, energy gap and ground states] \label{thm:asympt-1}Let $p>1$ and $J>0$ be fixed with either $1<p\le 2$ and $J$ is arbitrary or $p>2$ and $J<J_p$, and assume that $e$ in \eqref{enpersite} has a unique minimizer $h^\star $.
 Let $$E_N^\star =\min\{ E_N(\underline{\sigma}): \underline{\sigma}\in \{\pm 1\}^{\Lambda_N}\}$$ be the ground state energy on $\Lambda_N$.  Then the following statements hold.

\smallskip
{\rm (a) (General lower bound)} There exists a constant $c$ such that
     \begin{equation*} E_N(\underline{\sigma})\ge\sum_\mu h_\mu e(h_\mu)-cN^{-p} 
   \end{equation*}
for every $N$ and for every state $\underline{\sigma}$ with corresponding array $\{h_\mu\}$.
   
\smallskip
{\rm (b) (Asymptotics)}
   \begin{equation*}
        \lim_{N\to \infty}\frac{E_N^\star }{N}=e(h^\star ).
   \end{equation*}
   
\smallskip
{\rm (c) (Ground states are piecewise periodic)} There exists a $K_0>0$ such that for a ground state $\underline{\sigma}\equiv\underline{h}$
   \begin{equation*}
       \sum_{h_\mu\ne h^\star }h_\mu\le K_0.
   \end{equation*}
   
\smallskip
{\rm (d) (Energy gap)} If $N$ is a multiple of $2h^\star $ then the unique ground state $\underline{\sigma}$ has  $\underline{h}^\star $ as the corresponding array, and for every other state
   \begin{equation}
       E_N(\underline{\sigma})-E^\star _N= E_N(\underline{\sigma})-N e(h^\star )>\Delta,
   \end{equation}
   for a $\Delta>0$ independent of $N$.
  
 \end{thm}

 \section{Behaviour of renormalized energies}
In this section we consider the regimes when $e$ has only one minimizer
as in Theorem \ref{thm:asympt-1}.

We define the {\em renormalized functionals} 
\begin{equation}
F_N(\underline{\sigma})=E_N(\underline{\sigma})-Ne(h^\star )
\end{equation} for $\underline{\sigma}\in\Lambda_N$. 
By Theorem \ref{thm:asympt-1}(d) $F_N$ is non-negative, and strictly positive if $N$ is not a multiple of $2h^\star $.

A sequence $(\underline{\sigma}_N)_N$ with $\underline{\sigma}_N\in\Lambda_N$ is \emph{equibounded in energy} if $F_N(\underline{\sigma}_N)\le C$ for a constant $C$ independent of $N$.
Theorem \ref{thm:asympt-1}(c) guarantees that for every such sequence there exists $K_0$ such that 
\begin{equation}
\sum_{h^N_\mu\ne h^\star } h^N_\mu\le K_0,
\end{equation}
and  in particular that only a finite number of $h^N_\mu$ , uniformly bounded in $N$, can be different from $h^\star $.

If $N$ is a multiple of $2h^\star $, then $F_N$ is always nonnegative and is actually zero only when all blocks have length $h^\star $. We now show that if $N$ is not a multiple of $2h^\star $ then $F_N$ is strictly positive, uniformly in $N$.

\medskip
\begin{lemma}\label{lemma:strictlypositive} 
If $\underline{\sigma}$ is not equal to a state corresponding to $(h^\star ,\ldots,h^\star )$ then 
    \begin{equation*}
        F_N(\underline{\sigma})\ge \tilde\Delta
    \end{equation*}
   with $\tilde \Delta=\Delta/2h^\star $ and $\Delta$ provided in Theorem {\rm\ref{thm:asympt-1}(d)}. 
In particular, if $N$ is not a multiple of $2h^\star $ this lower bound is satisfied by all $\underline{\sigma}$.
\end{lemma}

\begin{proof}
Let $\underline{\sigma}\in \{\pm 1\}^{\Lambda_N}$, which we extend by periodicity so as to regard it as an element of $\{\pm 1\}^{\Lambda_{2h^\star N}}$. Then we have
\begin{equation}
    \begin{split}
        2h^\star  F_N(\underline{\sigma}) &= 2h^\star  (E_N(\underline{\sigma})-Ne(h^\star ))\\
&=E_{ 2h^\star  N}(\underline{\sigma})-2h^\star  Ne(h^\star )
\ge
\Delta,    \end{split}
\end{equation}
where we have used \eqref{5} in the second inequality and Theorem \ref{thm:asympt-1}(d) with $2h^\star  N$ in the place of $N$.
\end{proof}

In order to describe the behaviour of sequences of states with equibounded energy we identify $\Lambda_N$ with $(0,1]\cap \frac1N\mbb Z$ by scaling and define a convergence of spin states to piecewise-continuous $1$-periodic functions taking values in a set parameterized by the ground states themselves; that is, $\mbb Z/2h^\star \mbb Z$, or $\{1,\ldots, 2h^\star \}$.

\begin{defn}[convergence of spin states to piecewise-continuous functions]
Let $\underline{\sigma}\in \{\pm1\}^{\Lambda_N}$. In this section we first parameterize $\Lambda_N$ as $\{1,\ldots, N\}$. The state $\underline{\sigma}$ is then determined by a translation $\tau\in\Lambda_N$ and an array $(\underline{D}_1, \underline{G}_1,  \ldots , \underline{D}_S, \underline{G}_S)$,  where $\underline{G}_i$ is
a maximal sequences of consecutive pairs of blocks of length $h^\star $ starting with the plus sign, the arrays $\underline{D}_i$, which we call \emph{defects}, are arbitrary. For each $\underline{G}_i$ we let $R_i \in \Lambda_N$ denote its first element. After this decomposition to such $\underline{\sigma}$ we can associate a 
function $\widehat r\colon: (0,N]\to \mbb Z/2h^\star \mbb Z$ in the following way.  If $x\in (0,N]$ and $\lfloor x\rfloor-\tau\in\underline{G}_i$ then $\widehat r(x)$ is the element in $\mbb Z/2h^\star \mbb Z$ corresponding to $R_i$ $\operatorname{mod }2h^\star $. Finally, the function $r\colon \mbb T=\mbb R/\mbb Z\to \mbb Z/2h^\star \mbb Z$ 
is defined by $r(x)= \widehat r(Nx)$ in $(0,1]$, identified as $\mbb R/\mbb Z$.
Note that $r$ is a piecewise-constant function with at most $S$ discontinuity points in $(0,1]$. 
   
Let $\underline{\sigma}^N\in \{\pm1\}^{\Lambda_N}$ be a sequence with equibounded number of defects in the terminology above, and let $r^N\colon\mbb T\to \mbb Z/2h^\star \mbb Z$ be the corresponding functions. Then we say that $\underline{\sigma}^N$ {\em converge to} $r\colon \mbb T=\mbb R/\mbb Z\to \mbb Z/2h^\star \mbb Z$ if there are a finite number of points $x_1,\ldots, x_S\in (0,1]$ such that $r^N(x)= r(x)$ definitely on any compact interval of $(x_{k-1}, x_k)$ and we have set $x_0=x_S-1$. Note that this convergence is compact by the Bolzano--Weierstrass Theorem and it can be seen as an $L^p(0,1)$-convergence for any finite $p>1$ if we identify $\mbb Z/2h^\star \mbb Z$ with $\{1,\ldots,2h^\star \}$. Note moreover that $x_k$ may not be a discontinuity point for $r$ because we may have $R_k=R_{k-1}$ in the notation above.
\end{defn}


We now prove the key technical lemma stating that, in a configuration which is almost periodic in the sense that it is composed of all arrays of a common length $h$ except for a finite number of defects, the interaction between distant defects is negligible.  The lemma holds for any period $h$, but we will apply it only in the case of period $h^\star $.

\begin{lemma}[Decoupling of defects]\label{lemma} Let $D_0$ be fixed. For any $N$ let $\underline{\sigma}\in\{\pm1\}^{\Lambda_N}$ be a state with $S$ defects alternated with $h$-periodic parts, say 
\[\underline{\sigma}\sim (\underline{D}_1, \underline{P}_1,\ldots, \underline{D}_S,\underline{P}_S),\] 
where $\underline{P}_i=(\smash{\overbrace{h,\ldots, h}^{M_i \text{ times}}})$, with  $M_i$ even, we also suppose with no loss of generality that all periodic parts begin with the same sign. 

We define $M=\sum_i M_i$, $D=\sum_i D_i$, where $D_i$ is the total length of the defect $\underline{D}_i$. Thus, for the total number of sites $N$ we have $N=Mh+D$. We suppose that $D\le D_0$; apart from this restriction the $\underline{D}_i$ are arbitrary. 

We consider a state $\underline{\sigma}'$ obtained by removing a defect and substituting it by a $h$-periodic part:
\[\underline{\sigma}'\sim(\underline{D}_1, \underline{P}_1,\ldots, \underline{D}_{S-1},\underline{P}_{S-1}') ,\]
with $\underline{P}_{S-1}'$ containing $M'_{S-1}=M_{S-1}+M_S+2\lfloor D_S/2h\rfloor$ $h$-blocks. Note that the total number of sites in $\underline{\sigma}'$ is $N'=(M+2\lfloor D_S/2h\rfloor)h+ D-D_S $.

Then there exists a constant $C$ depending only on $D_0$ and not on $N$, such that the energy difference between $\underline{\sigma},\underline{\sigma}'$ can be estimated as
\begin{equation}\label{eq:decoupling}
\big|E_N(\underline{\sigma} )-E_{N'}(\underline{\sigma}')-\big(E_{N_S}( \underline{\delta}_S)-E_{M_S''h}(h,\ldots, h)\big)\big|\le C\min_i\{M_i\}^{1-p},
\end{equation}
where
\[\underline{\delta}_S\sim (\underline{D}_S, \underline{P}'_S)\]
is a state where only the defect $\underline{D}_S$ is present and the period part $\underline{P}_S'$ contains $M'_S=M+ 2\lfloor (D-D_S)/2h\rfloor$ $h$-blocks; thus $\underline{\delta}_S$ has $N_S=M'_Sh+D_S$ sites in total. The last energy is the energy of a completely periodic state with $M_S''=M+2\lfloor D_S/2h\rfloor+2\lfloor (D-D_S)/2h\rfloor$ $h$-periodic blocks; that by definition  is also equal to $M_S''he(h)$.
\end{lemma}
\begin{proof}
We note that the ferromagnetic parts of the left-hand side and of the right-hand side of \eqref{eq:decoupling} concide, so that we only have to check the equality for the antiferroagnetic parts.

We first focus on the antiferromagnetic part of the left-hand side. If  $\mu, \nu$ are any two distinct blocks of the state $\underline{\sigma}$, we denote by $d_{\mu,\nu}=h_{\mu+1}+\cdots+ h_{\nu-1}$, $d_{\nu,\mu}=2N-d_{\mu,\nu}-h_\mu-h_\nu=h_{\nu+1}+\cdots +h_{\mu-1}$ their distances along the two orientations of the torus. We denote by $\eaf_{\mu,\nu;N}(\underline\sigma)$ the anti-ferromagnetic energy among the two blocks of indices $\mu,\nu$. 

Along the same lines of Giuliani, Lebowitz and Lieb \cite{GLL} it is possible to obtain a closed form expression for $\eaf_{\mu,\nu;N}(\underline\sigma)$ by using the equality
\begin{equation}
    \frac{{\Gamma(p)} }{x^p}=\int_0^{+\infty}\alpha^{p-1}\mr e^{-\alpha x}\mr d\alpha,
\end{equation}
where $\Gamma$ is the Euler Gamma function, for each antiferromagnetic interaction term and using multiple geometric sums:
\begin{equation}\label{eq:eaf}
    \begin{split}
        \eaf_{\mu,\nu;N}(\underline\sigma) &\coloneqq (-1)^{\mu+\nu} \sum_{\substack{ 1\le i\le h_\mu\\ 1\le j\le h_\nu  }}\sum_{n\in \mbb Z} \frac{1}{|j+i+d_{\mu,\nu}-1+nN|^p }\\
        &= (-1)^{\mu+\nu} \int_0^{+\infty}\alpha^{p-1} \sum_{\substack{ 1\le i\le h_\mu\\ 1\le j\le h_\nu  }} \sum_{n\ge 0}\big(\mr e^{-\alpha(j+i+d_{\mu,\nu}-1+nN)}+\mr e^{-\alpha(j+i+d_{\nu,\mu}-1+nN)}\big)\frac{\mr d \alpha}{\Gamma(p)}\\
               &= (-1)^{\mu+\nu} \int_0^{+\infty}\alpha^{p-1} \bigg( \sum_{\substack{ 1\le i\le h_\mu\\ 1\le j\le h_\nu  }}\mr e^{-\alpha(j+i)}\bigg)\frac{\mr e^{-\alpha(d_{\mu,\nu}-1)}+\mr e^{-\alpha(d_{\nu,\mu}-1)}}{1-\mr e^{-\alpha N}}\frac{\mr d \alpha}{\Gamma(p)}\\
        & =(-1)^{\mu+\nu} \int_{0}^{+\infty}{ \frac{\mr e^{-\alpha}\alpha^{p-1}}{(1-\mr e^{-\alpha})^2}}\frac{(1-\mr e^{-\alpha h_\nu})(1-\mr e^{-\alpha h_\mu})\big(\mr e^{-\alpha d_{\mu,\nu}}+\mr e^{-\alpha d_{\nu,\mu}}\big)}{1-\mr e^{-\alpha N}}\frac{\mr d \alpha}{\Gamma(p)}.
    \end{split}
\end{equation}
Analogous expressions hold for $\underline{\sigma}'$ with $N'$ in place of $N$ and the correct block lengths and inter-block distances (which we will denote with a prime), as well as for $\underline{\delta}_S$ and the completely periodic state. Thus, the expressions for $\eaf_{\mu,\nu;N}(\underline\sigma)$ and $\eaf_{\mu,\nu;N'}(\underline\sigma ')$ is the same except for
\[\frac{ \mr e^{-\alpha d_{\mu,\nu}}+\mr e^{-\alpha d_{\nu,\mu}} }{1-\mr e^{-\alpha N}}\]
which becomes
\begin{equation*}
   \frac{ \mr e^{-\alpha d_{\mu,\nu}}+\mr e^{-\alpha d_{\nu,\mu}'} }{1-\mr e^{-\alpha N'}},
   \end{equation*}
since one of the block lengths $d_{\mu,\nu},d_{\nu,\mu}$ must stay the same both in $\underline{\sigma}$ and in $\underline{\sigma}'$, without loss of generality we suppose $d_{\mu,\nu}=d_{\mu,\nu}'$. We also remark that $\underline{\sigma}$ and $\underline{\sigma}'$ have the same block length except possibly for those blocks in $\underline{D}_S$ and the corresponding $h$-blocks that subistitute them in $\underline{\sigma}'$. Similarly, $\underline{\sigma}$ and $\underline{\delta}_S$ have the same block lengths except for the blocks of $\underline{D}_1,\ldots,\underline{D}_{S-1}$ and the $h$-blocks that substitute them in $\underline{\delta}_S$.

This last remark is helpful in view of the fact that equation (\ref{eq:decoupling}) can be written as
\begin{equation}
      E_N(\underline{\sigma} )-E_{N_S}( \underline{\delta}_S)=E_{N'}(\underline{\sigma}')-E_{M_S''h}(h,\ldots, h)+\mr O({\min\{M_S, M_{S-1},M-M_{S-1}-M_S\}^{1-p}}),
\end{equation}
and indeed we will compare some terms in $ E_N(\underline{\sigma} )$ with the corresponding terms in $E_{N'}(\underline{\sigma}')$, and some others with the corresponding terms in $E_{N_S}( \underline{\delta}_S)$.

In any case, we can sum geometrically over blocks in a periodic part $\underline{P}_i$: for a fixed block $\nu$ outside $\underline{P}_i$
\begin{equation}\label{eq:energy_single_block_with_periodic}
    \begin{split}
      &\sum_{\mu\in \underline{P}_i}(-1)^{\mu+\nu} \int_{0}^{+\infty}{ \frac{\mr e^{-\alpha}\alpha^{p-1}}{(1-\mr e^{-\alpha})^2}}\frac{(1-\mr e^{-\alpha h_\nu})(1-\mr e^{-\alpha h})}{1-\mr e^{-\alpha N}} \mr e^{-\alpha d_{\mu,\nu}}\frac{\mr d \alpha}{\Gamma(p)}\\
           & =\sum_{i=0}^{M_i-1}(-1)^{i+\nu} \int_{0}^{+\infty}{\frac{\mr e^{-\alpha}\alpha^{p-1}}{(1-\mr e^{-\alpha})^2}}\frac{(1-\mr e^{-\alpha h_\nu})(1-\mr e^{-\alpha h})}{1-\mr e^{-\alpha N}} \mr e^{-\alpha (d(i)_\nu+ih)}\frac{\mr d \alpha}{\Gamma(p)}\\
            &=(-1)^\nu\int_{0}^{+\infty}{ \frac{\mr e^{-\alpha}\alpha^{p-1}}{(1-\mr e^{-\alpha})^2}}\frac{(1-\mr e^{-\alpha h_\nu})(1-\mr e^{-\alpha h})}{1-\mr e^{-\alpha N}}\frac{1-\mr e^{-\alpha M_i h}}{1+\mr e^{-\alpha h}}\mr e^{-\alpha d(i)_\nu}\frac{\mr d \alpha}{\Gamma(p)},
    \end{split}
\end{equation}
where $d(i)_\nu=\min_{\mu\in \underline{P}_i}\{d_{\mu,\nu}\}$;  analogous expressions hold {\em mutatis mutandis} for the other three states. Then, we can sum over blocks belonging to some other $\underline{P}_j$, with $j\ne i$:
\begin{equation}\label{eq:PiPj}
    \begin{split}
         \sum_{\mu\in \underline{P}_i, \nu \in \underline{P}_j}E^{\operatorname{af}}_{\mu,\nu;N}(\underline{\sigma}) &= \int_{0}^{+\infty}{ \frac{\mr e^{-\alpha}\alpha^{p-1}}{(1-\mr e^{-\alpha})^2}}\frac{ (1-\mr e^{-\alpha h})^2(\mr e^{-\alpha d{(i,j)}}+\mr e^{-\alpha d(j,i)})}{1-\mr e^{-\alpha N}}\frac{1-\mr e^{-\alpha M_i h}}{1+\mr e^{-\alpha h}}\frac{1-\mr e^{-\alpha M_j h}}{1+\mr e^{-\alpha h}}\frac{\mr d \alpha}{\Gamma(p)},
    \end{split}
\end{equation}
where $d{(i,j)}=\min_{\mu\in \underline{P}_i, \nu\in \underline{P}_j}\{d_{\mu,\nu}\}$. A corresponding expression holds for $\underline{\sigma}'$.

In order to estimate the differences  $\eaf_{\mu,\nu}(\underline\sigma)-\eaf_{\mu,\nu}(\underline\sigma')$ we note that, since $N>N'$, we have
\begin{equation}\label{eq:factor}
\begin{split}
\bigg|\frac{1}{1-\mr e^{-\alpha N}}-\frac{1}{1-\mr e^{-\alpha N'}}\bigg|
&=\bigg|\frac{\mr e^{-\alpha N}-\mr e^{-\alpha N'}}{(1-\mr e^{-\alpha N})(1-\mr e^{-\alpha N'})}\bigg|\\
&=\mr e^{-\alpha N'}\frac{1-\mr e^{-\alpha (N-N')}}{(1-\mr e^{-\alpha N})(1-\mr e^{-\alpha N'})}\\
   &\le\frac{\mr e^{-\alpha N'} }{ 1-\mr e^{-\alpha N'}},\\
\end{split}
\end{equation}
with analogous estimates holding for the other couples of states.
Then, for $i\ne j$, if $i$ and $j$ are not exactly $S-1, S$ and $S\ge 3$, we estimate
\begin{equation}
    \begin{split}
        &\bigg| \sum_{\mu\in \underline{P}_i, \nu \in \underline{P}_j} E^{\operatorname{af}}_{\mu,\nu; N}(\underline{\sigma})- E^{\operatorname{af}}_{\mu,\nu;N'}(\underline{\sigma}')\bigg|\\
        &   \le\bigg| \int_{0}^{+\infty} \frac{\mr e^{-\alpha}\alpha^{p-1}}{(1-\mr e^{-\alpha})^2}\frac{ (1-\mr e^{-\alpha h})^2(\mr e^{-\alpha (N'+d(i,j))}+\mr e^{-\alpha (N'+d(j,i))}+\mr e^{-\alpha d(j,i)}-\mr e^{-\alpha d'(j,i)})}{1-\mr e^{-\alpha N'}}\\
        &\mkern550mu\cdot\frac{1-\mr e^{-\alpha M_i h}}{1+\mr e^{-\alpha h}}\frac{1-\mr e^{-\alpha M_j h}}{1+\mr e^{-\alpha h}}\frac{\mr d \alpha}{\Gamma(p)}\bigg|\\
        &\le h^2\int_{0}^{+\infty}{\mr e^{-\alpha}\alpha^{p-1}}  {  (\mr e^{-\alpha (N'+d(i,j))}+\mr e^{-\alpha (N'+d(j,i))}+\mr e^{-\alpha d(j,i)}+\mr e^{-\alpha d'(j,i)}) \frac{\mr d \alpha}{\Gamma(p)}},
    \end{split}
\end{equation}
and conclude noting that in this case $ d(j,i), d'(j,i)\ge \min\{M_S,M_{S-1}\}h$. Analogously one estimates the corresponding differences in the contributions of $\underline{\delta}_S$ and the totally periodic state.

If $i,j$ are $S-1,S$ we compare instead
\begin{equation}
    \begin{split}
        &\bigg| \sum_{\mu\in \underline{P}_{S-1}, \nu \in \underline{P}_S} E^{\operatorname{af}}_{\mu,\nu; N}(\underline{\sigma})- E^{\operatorname{af}}_{\mu,\nu;N_S}(\underline{\delta}_S)\bigg|\\
        &   \le\bigg| \int_{0}^{+\infty}\frac{\mr e^{-\alpha}\alpha^{p-1}}{(1-\mr e^{-\alpha})^2}\frac{ (1-\mr e^{-\alpha h})^2(\mr e^{-\alpha (N_S+d(S-1,S))}+\mr e^{-\alpha (N_S+d(S,S-1))}+\mr e^{-\alpha d(S,S-1)}-\mr e^{-\alpha d_S(S,S-1)})}{1-\mr e^{-\alpha N_S}}\\
        &\hskip-1.9cm\mkern640mu\cdot\frac{1-\mr e^{-\alpha M_i h}}{1+\mr e^{-\alpha h}}\frac{1-\mr e^{-\alpha M_j h}}{1+\mr e^{-\alpha h}}\frac{\mr d \alpha}{\Gamma(p)}\bigg|\\
        &\le h^2\int_{0}^{+\infty} {\mr e^{-\alpha}\alpha^{p-1}}  {  (\mr e^{-\alpha (N'+d(S-1,S))}+\mr e^{-\alpha (N_S+d(S,S-1))}+\mr e^{-\alpha d(S,S-1)}+\mr e^{-\alpha d_S(S,S-1)})\frac{\mr d \alpha}{\Gamma(p)}},
    \end{split}
\end{equation}
since for $\underline{\sigma}$ and $\underline{\delta}_S$ is the distance $d(S-1,S)$ which is the same for the two of them, and also in this case we conclude. If $S=2$ the reasoning is very similar, the only difference being that one must split $ \sum_{\mu\in \underline{P}_{S-1}, \nu \in \underline{P}_S, \mu\ne\nu} $ into contributions with $d(2,1)$ and those with $d(1,2)$, the former must be compared with corresponding contributions in $\underline{\sigma}'$ (which in this case could be called $\underline{\delta}_1$), the latter with $\underline{\delta}_2$.

In the end we are left with contribution relative a $\mu,\nu$ belonging to the same periodic part: as for contributions where one block belongs to a defect or to an $h$-block substituting it we estimate them by \eqref{eq:energy_single_block_with_periodic} and by comparison between couple of states where the block in question is the same.

For $\mu,\nu\in \underline{P}_i, \nu\ne \mu$ we obtain a formula a little different than before:
\begin{equation}\label{eq:same_periodic_part}
    \begin{split}
         &\sum_{\mu,\nu\in \underline{P}_i,  \mu\ne\nu} E^{\operatorname{af}}_{\mu,\nu;N}(\underline{\sigma}) \\
         &
         = \sum_{i=0}^{M_i-1}\sum_{j=i+1}^{M_i-1}(-1)^{i+j}\int_{0}^{+\infty}{
         \frac{\mr e^{-\alpha}\alpha^{p-1}}{(1-\mr e^{-\alpha})^2}}\frac{ (1-\mr e^{-\alpha h})^2\mr e^{-\alpha (j-1)h}}{1-\mr e^{-\alpha N}}(1+\mr e^{-\alpha(N-M_ih))})\frac{\mr d \alpha}{\Gamma(p)}\\
          &
         = -\sum_{i=0}^{M_i-1}\int_{0}^{+\infty}{\frac{\mr e^{-\alpha}\alpha^{p-1}}{(1-\mr e^{-\alpha})^2}}\frac{ (1-\mr e^{-\alpha h})^2}{1-\mr e^{-\alpha N}}\,\mr e^{-\alpha ih}\frac{1-(-1)^{M_i-i}\mr e^{-\alpha(M_i-i)h}}{1+\mr e^{-\alpha h}}(1+\mr e^{-\alpha(N-M_ih))})\frac{\mr d \alpha}{\Gamma(p)}\\
         & =- \sum_{i=0}^{M_i-1}\int_{0}^{+\infty}{ \frac{\mr e^{-\alpha}\alpha^{p-1}}{(1-\mr e^{-\alpha})^2}}\frac{ (1-\mr e^{-\alpha h})^2}{1-\mr e^{-\alpha N}}\frac{\mr e^{-\alpha ih}}{1+\mr e^{-\alpha h}}(1+\mr e^{-\alpha(N-M_ih))})\frac{\mr d \alpha}{\Gamma(p)}\\
                  & =- \int_{0}^{+\infty}{ \frac{\mr e^{-\alpha}\alpha^{p-1}}{(1-\mr e^{-\alpha})^2}}\frac{ (1-\mr e^{-\alpha h})^2}{1-\mr e^{-\alpha N}}\frac{1-\mr e^{-\alpha M_ih}}{(1+\mr e^{-\alpha h})(1-\mr e^{-\alpha h})}(1+\mr e^{-\alpha(N-M_ih))})\frac{\mr d \alpha}{\Gamma(p)}.
    \end{split}
\end{equation}
We can then estimate
\begin{equation}\label{eq:same_periodic_part2}
    \begin{split}
         &\bigg|\sum_{\mu,\nu\in \underline{P}_i,  \mu\ne\nu} \eaf_{\mu,\nu;N}(\underline{\sigma})-\eaf_{\mu,\nu;N'}(\underline{\sigma}')\bigg| \\
                  & \le M_ih^2 \int_{0}^{+\infty} {\mr e^{-\alpha}\alpha^{p-1} }\bigg|\frac{1+\mr e^{-\alpha(N-M_ih)}}{1-\mr e^{-\alpha N}}-\frac{1+\mr e^{-\alpha(N'-M_ih)}}{1-\mr e^{-\alpha N'}}\bigg|\frac{\mr d \alpha}{\Gamma(p)},
    \end{split}
\end{equation}
and conclude by observing that
\begin{equation}\label{eq:int_1}
    \begin{split}
                  & \int_{0}^{+\infty} {\mr e^{-\alpha}\alpha^{p-1} }\bigg|\frac{\mr e^{-\alpha(N-M_i h)}}{1-\mr e^{-\alpha N}}-\frac{\mr e^{-\alpha(N'-M_ih)}}{1-\mr e^{-\alpha N'}}\bigg|\frac{\mr d \alpha}{\Gamma(p)}\\
  & =\int_{0}^{+\infty} {\mr e^{-\alpha}\alpha^{p-1} }\bigg(\frac{\mr e^{-\alpha(N'-M_ih)}}{1-\mr e^{-\alpha N'}}-\frac{\mr e^{-\alpha(N-M_i h)}}{1-\mr e^{-\alpha N}}\bigg)\frac{\mr d \alpha}{\Gamma(p)}\\
  &\le   \int_{0}^{+\infty} {\mr e^{-\alpha}\alpha^{p-1} }\mr e^{-\alpha(N'-M_i h)}\bigg(\frac{1}{1-\mr e^{-\alpha N'}}-\frac{1}{1-\mr e^{-\alpha N}}\bigg)\frac{\mr d \alpha}{\Gamma(p)}\\
  &\le   \int_{0}^{+\infty} {\mr e^{-\alpha}\alpha^{p-1} } \bigg(\frac{1}{1-\mr e^{-\alpha N'}}-\frac{1}{1-\mr e^{-\alpha N}}\bigg)\frac{\mr d \alpha}{\Gamma(p)}\\
    \end{split}
\end{equation}
and that
\begin{equation}\label{eq:int_2}
    \begin{split}
                  & \int_{0}^{+\infty}
                  {\mr e^{-\alpha}\alpha^{p-1} }\mr e^{-\alpha(N-N')}\Big|\frac{1}{1-\mr e^{-\alpha N}}-\frac{1}{1-\mr e^{-\alpha N'}}\Big|\frac{\mr d \alpha}{\Gamma(p)}\\
                  &\le \int_{0}^{+\infty} {\mr e^{-\alpha}\alpha^{p-1} }\frac{\mr e^{-\alpha N'}}{1-\mr e^{-\alpha N'}}\frac{\mr d \alpha}{\Gamma(p)}\\
                  &=(N')^{-p}\frac1{\Gamma(p)}\int_{0}^{+\infty}
                  {\mr e^{-t/N'}t^{p-1} }\frac{\mr e^{-t}}{1-\mr e^{-t}}{\mr d t},
    \end{split}
\end{equation}
since both integrals in $t$ tend to convergent integrals: in (\ref{eq:int_1}) to $\displaystyle\int_0^{+\infty}\frac{e^{-t}}{1-\mr e^{-t}}\mr  t^{p-1}d t$  while  in (\ref{eq:int_2}) to $\displaystyle\int_{0}^{+\infty}\frac{\mr e^{-t}}{1-\mr e^{-t}}
                 t^{p-1}  {\mr d t}$.

\smallskip
We are now left with the very last case: $\mu=\nu$.
\begin{equation}\label{eq:same_block}
    \begin{split}
        E^{\operatorname{af}}_{\mu,\mu;N}(\underline{\sigma})&=\overbrace{\sum_{1\le k <h_\mu}\frac{h-k}{k^p}}^{\eqqcolon  \widetilde{\eaf}(h)}+\sum_{\substack{n> 0\\ 1\le i,j\le h }} \frac{1}{(j+i-1-h+nN)^p}\\
        &= \widetilde{\eaf}(h)+ \int_{0}^{+\infty}\frac{\mr e^{-\alpha}\alpha^{p-1}}{(1-\mr e^{-\alpha})^2}\frac{(1-\mr e^{-\alpha h})^2}{1-\mr e^{-\alpha N}}\mr e^{-\alpha (N-h)}\frac{\mr d\alpha}{\Gamma(p)}.
    \end{split}
\end{equation}
The first summand also admits an integral representation, but this is irrelevant to this lemma since it only matters that it is independent of $N$. The second summand can instead be treated by reasoning as before, thus getting 
\begin{equation}
    \big| E^{\operatorname{af}}_{\mu,\mu;N}(\underline{\sigma})- E^{\operatorname{af}}_{\mu,\mu;N'}(\underline{\sigma}')\big|\le CN^{-p}.
\end{equation}
Summing over all $\mu$ we get a factor $O(N)$, and thus the final estimate is $O(N^{1-p})$.
\end{proof}

\begin{rmk}\label{rmk:localization}
    Applying Lemma \ref{lemma} iteratively $S-1$ times to defects $D_S, D_{S-1},\ldots, D_2$, in order to decouple all of them, one gets
    \begin{equation}
    E_N(\underline{\sigma}) = \sum_{i=1}^S E_{N_i}(\underline{\delta}_i)-\sum_{i=2}^{S} E_{M''_ih}(h,\ldots, h)+\mr O(\min_i M_i)^{1-p}
    \end{equation}
    where 
    \begin{equation}
        \begin{split}
            N_i&=Mh+2h\sum_{j=i+1}^S\lfloor D_j/2h\rfloor +D_i+ 2h\bigg\lfloor\sum_{j=1}^{i-1}D_j/2h\bigg\rfloor, \\
            M''_i&=Mh+2h\sum_{j=i}^S\lfloor D_j/2h\rfloor+ 2h\bigg\lfloor\sum_{j=1}^{i-1}D_j/2h\bigg\rfloor .
        \end{split}
    \end{equation} 
If now we specialize to the case of $h=h^\star $, subtracting $Ne(h^\star )$ to both sides we get
    \begin{equation*}
    F_N(\underline{\sigma}) = \sum_{i=1}^S F_{N_i}(\underline{\delta}_i)+\mr O(\min_i M_i)^{1-p}.
    \end{equation*}
Indeed, for $i=2,\ldots, S$
\begin{equation}\label{eq:LHS1}
    \begin{split}
         E_{N_i}(\underline{\delta}_i)-E_{M''_ih}(h,\ldots, h)&=E_{N_i}(\underline{\delta}_i)-M''_ie(h) \\
         &=E_{N_i}(\underline{\delta}_i)-M''_ie(h^\star ) \pm N_i e(h^\star )  \\
         &=F_{N_i}(\underline{\delta}_i)+\big(D_i-2h^\star \lfloor D_i/2h\rfloor\big)e(h^\star ),
    \end{split}
\end{equation}
\begin{equation}\label{eq:LHS2}
    \begin{split}
         E_{N_1}(\underline{\delta}_i) &=F_{N_1}(\underline{\delta}_i)+N_1 e(h^\star )\\
         &=F_{N_1}(\underline{\delta}_i)+ \Big(Mh^\star +2h^\star \sum_{j\ge 2}\lfloor D_j/2h^\star \rfloor+D_1\Big)e(h^\star ).
         \end{split}
\end{equation}
Thus, recalling that $N=Mh^\star +D=Mh^\star +\sum_iD_i$, the sum of (\ref{eq:LHS1}) and (\ref{eq:LHS2}) is exactly
\begin{equation}
    \sum_{i=1}^SF_{N_i}(\underline{\delta}_i)+Ne(h^\star ).
\end{equation}
\end{rmk}

\begin{lemma}[Localization]\label{lemma:localization}
Given an array $\underline{D}$ of length $D$, let $\underline{\delta}$ be a state of the form $(\underline{D},\underline{P})$ with  $\underline{P}$ made of $M$ $h$-blocks and let $\underline{\delta}'$ be a state of the form $(\underline{D},\underline{P}')$ with the same $\underline{D}$ but $\underline{P}'$ made of $M'>M$ $h$-blocks.  Then,
 \begin{equation}
     \big|E_N(\underline{\delta})-E_{N'}(\underline{\delta}')-\big( E_{Mh}(h,\ldots,h)-E_{M'h}(h,\ldots, h)\big)\big|\le CM^{-p}M',
 \end{equation}
 where $N=Mh+D$, $ N'=M'h+D$ and the constant $C$ is independent of $N$. Specializing to $h=h^\star$ and adding and subtracting $De(h^\star)$ the previous equation becomes
 \begin{equation}
     \big|F_N(\underline{\delta})-F_{N'}(\underline{\delta}')\big|\le C M^{-p}M'.
 \end{equation}
\end{lemma}
\begin{proof} As in the previous lemma, both sides have the same ferromagnetic energy. As for the antiferromagnetic energy, recalling (\ref{eq:same_periodic_part}), we can estimate
\begin{equation}
 \begin{split}
       &\bigg| \sum_{\mu,\nu\in \underline{P},\mu\ne \nu}\eaf_{\mu,\nu;N}(\underline{\delta})- \sum_{\mu,\nu\in \underline{P}',\mu\ne \nu}\eaf_{\mu,\nu;N'}(\underline{\delta}') \bigg|\\
    &\le  h^2 \int_0^{+\infty}\mr e^{-\alpha}\alpha^{p-1}\Big|\frac{1-\mr e^{-\alpha Mh}}{1-\mr e^{-\alpha h}}\frac{1+\mr e^{-\alpha D}}{1-\mr e^{-\alpha (D+Mh)}} -\frac{1-\mr e^{-\alpha M'h}}{1-\mr e^{-\alpha h}}\frac{1+\mr e^{-\alpha D}}{1-\mr e^{-\alpha (D+M'h)}}\Big|\frac{\mr d\alpha}{\Gamma(p)}\\
        &\le 2 h^2 \int_0^{+\infty}\mr e^{-\alpha}\alpha^{p-1}\Big|\frac{\mr e^{-\alpha M'h}-\mr e^{-\alpha Mh}+\mr e^{-\alpha(D+Mh)}-\mr e^{-\alpha(D+M'h)}}{(1-\mr e^{-\alpha h})(1-\mr e^{-\alpha (D+Mh)})(1-\mr e^{-\alpha (D+M'h)})}\Big|\frac{\mr d\alpha}{\Gamma(p)}\\
    &= 2 h^2 \int_0^{+\infty}\mr e^{-\alpha}\alpha^{p-1}\Big|\frac{(\mr e^{-\alpha M'h}-\mr e^{-\alpha Mh})(1-\mr e^{-\alpha D})}{(1-\mr e^{-\alpha h})(1-\mr e^{-\alpha (D+Mh)})(1-\mr e^{-\alpha (D+M'h)})}\Big|\frac{\mr d\alpha}{\Gamma(p)}\\
        &\le 2 D h^2 \int_0^{+\infty}\mr e^{-\alpha}\alpha^{p-1} \frac{\mr e^{-\alpha Mh}-\mr e^{-\alpha M'h}}{ (1-\mr e^{-\alpha (D+Mh)}) (1-\mr e^{-\alpha (D+M'h)})}\frac{\mr d\alpha}{\Gamma(p)}\\
    &= 2 Dh^2 \int_0^{+\infty}\mr e^{-\alpha}\alpha^{p-1}\mr e^{-\alpha Mh} \frac{1-\mr e^{-\alpha (M'-M)h}}{ (1-\mr e^{-\alpha (D+Mh)}) (1-\mr e^{-\alpha (D+M'h)})} \frac{\mr d\alpha}{\Gamma(p)}  \\
        &\le 2D h^2 \int_0^{+\infty}\mr e^{-\alpha}\alpha^{p-1}\frac{\mr e^{-\alpha Mh} }{1-\mr e^{-\alpha (D+Mh)}}\frac{\mr d\alpha}{\Gamma(p)}\\
    &=2 Dh^2 M^{-p} \int_0^{+\infty}\mr e^{-t/M}t^{p-1}\frac{\mr e^{-t h} }{1-\mr e^{-t (D+Mh)/M}}\frac{\mr dt}{\Gamma(p)}\\
    &\le c M^{-p}.
 \end{split}
\end{equation}
When $\mu=\nu$ we recall Equation (\ref{eq:same_block}) and then we proceed as above. The right-hand side is treated analogously.
\end{proof}

\smallskip
We now introduce a function
$\phi$ which has the meaning of a minimal interaction energy between two ground states with a relative translation of $j$. It will appear as an energy density for the $\Gamma$-limit.

\begin{defn}[anti-phase energy density] For any $j\in\{1,\ldots, 2h^\star \}$, which we identify also with $\mbb Z/2h^\star \mbb Z$, we define
\[\phi(j)=\lim_{\buildrel {N\to +\infty} \over {N\equiv j \,\hbox{\scriptsize mod.}\, 2h^\star  }}
(E^\star _{N}-N e(h^\star )),
\]
which exists thanks to Remark \ref{rmk:limitfuncmin}.
\end{defn}

\begin{rmk}\label{rmk:limitfuncmin}
(i) The limit the definition of $\phi$ exists. To check this, let $\underline\sigma_K$ be a minimizer for $E_{2Kh^\star +j}$. We note that, by Theorem \ref{thm:asympt-1}(c), $\underline\sigma_K$ is equivalent to an array of the form 
\begin{equation*}
   ( \underline D_{1},\underline P_{1}\ldots, \underline D_{S},\underline P_{S}),
\end{equation*}
with $\underline P_{i}$ arrays composed of an even number $M_K$ of $h^\star$-blocks, while the  $\underline D_{i}$ can be anything and both $S$ and the their total length $D$ are equibounded in $K$. With fixed $K$, let $K'>K$ and let 
\begin{equation*}
    \underline\sigma\sim  ( \underline D_{1},\underline P_{1}\ldots, \underline D_{S},\underline P_{S}')
\end{equation*}
be the test state for $E_{2K'h^\star +j}$ constructed by adding $h^\star$-blocks to $\underline P_{S}$ till the desired length. We can apply Lemma \ref{lemma:localization} to $\underline\sigma$ and $\underline{\sigma}_K$ considering $\underline{\delta}\coloneqq ( \underline D_{1},\underline P_{1}\ldots, \underline D_{S})$, obtaining
\[
\min F_{2K'h^\star +j}\le F_{2K'h^\star +j}(\underline\sigma)\le \min F_{2Kh^\star +j}+ o(1)
\]
as $K'\to+\infty$ and $K\to+\infty$. It then suffices to take the limsup as $K'\to+\infty$ first, and then the liminf as $K\to+\infty$.

 (ii) We have $\phi(0)=0$ and $\phi(j)>0$ if $j\not\equiv 0$ modulo $2h^\star $. This is again a consequence of Theorem  \ref{thm:asympt-1}(c).
 
 (iii) The function $\phi$ is subadditive; that is, $\phi(j+k)\le \phi(j)+\phi(k)$ for all $j,k$. To check this, for all $K$ let $\underline\sigma^j_K $ and $\underline\sigma^k_K $ be minimizers for $E_{2Kh^\star+j}$ and $E_{2Kh^\star+k}$, respectively. Then, we can construct a test state $\underline\sigma= \underline\sigma_{K,K'}\sim (\underline\sigma^j_K, \underline P,\underline\sigma^k_K, \underline P )$ for $E_{4(K+K')h^\star+j+k}$ by inserting two equal sequences $\underline{P}$ of $2K'$ $h^\star$-blocks. By Lemma \ref{lemma}, we then have
\begin{equation*}
    \begin{split}
        \phi(j+k)&\le \liminf_{K'\to+\infty}F_{4(K+K')h^\star+j+k}(\underline\sigma_{K,K'})\\
        &\le \liminf_{M\to+\infty} (F_{2Mh^\star+j}(\underline\sigma^j_K,\underline{P}')+F_{2Mh^\star+k}(\underline\sigma^k_K,\underline{P}''))\\
&=\phi(j)+\phi(k)+o_K(1),
    \end{split}
\end{equation*}
where, with a small abuse of notation, we denoted by $\underline\sigma^j_K,\underline{P}'$ a state obtained by adding to $\underline\sigma^j_K$ the correct number of $h^\star$-block required by the lemma, and analogously for  $\underline\sigma^k_K$. The last equality follows by (i).
\end{rmk}

\begin{thm} Let $j\in\{1,\ldots, 2h^\star \}$; then there exists the $\Gamma$-limit
    \[\Gamma\hbox{-}\lim_{\buildrel {N\to +\infty} \over {N\equiv j \,\hbox{\scriptsize mod.}\, 2h^\star  }} F_N = F^j_{\infty},\] where $F^j_{\infty}$ is defined on piecewise-constant functions $r\colon \mbb T\to \mbb Z/2h^\star \mbb Z$ by
    \[
    F^j_{\infty}(r)=\begin{cases}\displaystyle\sum_{x\in J(r)}\phi(\Delta r(x)\big) & \hbox{if } \displaystyle\sum_{x\in J(r)}\Delta r(x)\equiv j \hbox{ modulo } 2h^\star \\
    +\infty &  \hbox{otherwise,}\end{cases}\] $J(r)$ is the set of discontinuity points of $r$ and $\Delta r(x)=r(x+)-r(x-)$ is the jump size at $x$. 
\end{thm}

\begin{proof} Let $j\in\{1,\ldots, 2h^\star\}$ be fixed, and consider $N\equiv j$ modulo $2h^\star$.
    Given $(\underline{\sigma}_N)_N$ equibounded and converging to some $r$, we can write, up to possibly passing to a subsequence in order to have $S$ independent of $N$, 
    \begin{equation*}
        \underline{\sigma}_N=(\underline{D}_{1;N}, \underline{H}_{1;N},\ldots, \underline{D}_{S;N},\underline{H}_{S;N}),
    \end{equation*}
    with the total length $D_N$ of the defects uniformly bounded, say by $D_0$, while the total length of $\underline{H}_{i;N}$, say $M_{i;N}h^\star$, goes to infinity for every $i$ (note indeed that if some $M_{i;N}$ stays bounded we can simply include it into an adjacent defect). Since $M_{i;N}$ are taken even, we have that $D_N\equiv j$ modulo $2h^\star$, which implies that
\begin{equation}\label{j-cond}
 \sum_{x\in J(r)}\Delta r(x)\equiv j \hbox{ modulo } 2h^\star.   
\end{equation}
Using the notation of Lemma \ref{lemma}, by Remark \ref{rmk:localization}, we have
    \begin{equation*}
        \begin{split}
            F_N(\underline{\sigma}_N)\ge \sum_{i=1}^S F_{N_i}(\underline{D}_i,\underline{H}'_{i;N})+\mr O (\min_{i}M_{i;N})^{1-p}.
        \end{split}
    \end{equation*}
    We can suppose, up to subsequences, that the length of $\underline D_{i,N}$ is converging to some $j_i$ modulo $2h^\star$,
    and, using $(\underline{D}_{i,N},\underline{H}'_{i;N})$ as a test state for $E^\star_{N_i}$ we have
      \begin{equation*}
        \begin{split}
            \liminf_{N\to+\infty} F_N(\underline{\sigma}_N)\ge \sum_{i=1}^S \phi(j_i).        \end{split}
    \end{equation*}  
    The desired lower bound now follows by the subadditivity of $\phi$, noting that, by the hypothesis of convergence to $r$,  for all $x\in J(r)$ we have that
    the sum of all $j_i$ corresponding to $x$ is equal to $\Delta r(x)$ modulo $2h^\star$.
%
%

\smallskip
To prove the upper bound for the $\Gamma$-limit, we
consider $r\colon \mbb T\to \mbb Z/2h^\star \mbb{Z}$, with a finite number $S$ of discontinuities, say at points $x_1, \ldots, x_S$ and of jump size $\Delta r(x_k)$, $k\in\{1,\ldots, S\}$, and satisfying \eqref{j-cond}. We fix $\eta>0$ and for each $k$ we take a state $\underline{\sigma}_k$ and $N_k\equiv \Delta r(x_k)$ modulo $2h^\star $ such that
\[
F_{N_k}(\underline{\sigma}_k)=E_{N_k}(\underline{\sigma}_k)- N_k e(h^\star )\le \phi(\Delta r(x_i))+\eta
\]
In the notation above, we identify $\underline{\sigma}_k$ with a defect $\underline{D}_k$, and for $N\equiv j$ modulo $2h^\star $ we can construct $\underline{\sigma}_N$ equivalent to $(\underline{D}_1,\underline{H}_{N,1},\ldots,\underline{D}_S,\underline{H}_{N,S})$ and such that  $\underline{\sigma}_N$ converges to $r$, with each $\underline{H}_{N,k}$ equal to an even number of $2h^\star $. This is possible thanks to condition \eqref{j-cond}. Note that we can assume that the length of each $\underline{H}_{N,k}$ tends to $+\infty$, so that again by Remark  \ref{rmk:localization} we have
 \[
\limsup_{N\to+\infty}F_N(\underline{\sigma}_N)\le   \sum_{k=1}^S \limsup_{N\to+\infty}
F_{N'_k}(\underline{D}_k,\underline{H}_{N,k})\le \sum_{k=1}^S \phi(\Delta r(x_i))+ S\eta
\]
of length equivalent to $\Delta r(x_k)$ modulo $2h^\star $. By the arbitrariness of $\eta>0$ this concludes the proof.
%
%
%
\end{proof}

\noindent{\bf Acknowledgments.} AB is a member of GNAMPA, INdAM, and partially supported by the MIUR Excellence Department Project
MatMod@TOV awarded to the Department of Mathematics, University of Rome Tor Vergata. FC is a member of GNFM, INdAM.

\bibliographystyle{abbrv}

\bibliography{references}

\end{document}